\documentclass[11 pt,a4paper]{article}

\usepackage{amsmath,amscd}
\usepackage{amstext}
\usepackage{amssymb,epsfig}
\usepackage{amsthm}
\usepackage{amsfonts}
\usepackage{graphicx}
\usepackage[all]{xy}
\usepackage{tikz}

\renewcommand{\footnote}{\endnote}
\newtheorem{theorem}{Theorem}[section]

\newtheorem{lemma}[theorem]{Lemma}
\newtheorem{proposition}[theorem]{Proposition}
\newtheorem{corollary}[theorem]{Corollary}

\theoremstyle{definition}
\newtheorem{definition}[theorem]{Definition}
\newtheorem{remark}[theorem]{Remark}

\usepackage[utf8]{inputenc}      

\begin{document}
\title{Entropy of polyhedral billiard}

\author{Nicolas Bédaride\footnote{ Laboratoire d'Analyse Topologie et Probabilités  UMR 7353 , Université Aix Marseille, avenue escadrille Normandie Niemen 13397 Marseille cedex 20, France. nicolas.bedaride@univ-amu.fr}}
\date{}

\maketitle
\begin{abstract}
We consider the billiard map in a convex polyhedron of
$\mathbb{R}^3$, and we prove that it is of zero topological entropy.
\end{abstract}

\section{Introduction}
A billiard ball, i.e.\ a point mass, moves inside a polyhedron $P$
with unit speed along a straight line until it reaches the
boundary $\partial{P}$, then it instantaneously changes direction
according to the mirror law, and continues along the new line.

Label the faces of $P$ by symbols from a finite alphabet
$\mathcal{A}$ whose cardinality equals the number of faces of $P$.
Consider the set of all billiard orbits. After coding, the set of all the words is 
a language. We define the complexity of the language, $p(n)$, by the number of
words of length $n$ that appears in this system. 
How complex is the game of billiard inside a polygon or a polyhedron?
For the cube the computations have been done, see \cite{ moi1, moi2}, but there is no result for a general polyhedron.
One way to answer this question is to compute the topological entropy of the billiard map. 

There are three different proofs that polygonal billiard have zero
topological entropy \cite{Ka,Ga.Kr.Tr,Gu.Ha}. Here we consider the
billiard map inside a polyhedron. 
We want to compute the topological entropy of the billiard map in a polyhedron.
The idea is to improve the proof of Katok. Thus we must compute the metric entropy of each ergodic measure.
When we follow this proof some difficulties appear. In particular
a non atomic ergodic measure for the related shift can
have its support included in the boundary of the definition set.
Such examples were known for some piecewise isometries of
$\mathbb{R}^2$ since the works of Adler, Kitchens and Tresser
\cite{Ad.Ki.Tr}; Goetz and Poggiaspalla \cite{Go,Go.Pog}.  Piecewise isometries and billiard are related since the first return map of the directional billiard flow inside a rational polyhedron is a piecewise isometry.

Our main result is the following
\begin{theorem}\label{entro}
Let $P$ be a convex polyhedron of $\mathbb{R}^3$ and let $T$ be the billiard map, then $$h_{top}(T)=0.$$
\end{theorem}
\begin{corollary}
The complexity of the billiard map satisfies $$\lim_{n\rightarrow+\infty}\frac{\log{p(n)}}{n}=0.$$
\end{corollary}
 
 For the standard definitions and properties of entropy we refer to Katok and Hasselblatt \cite{Ha.Ka}.

\subsection{Overview of the proof}
We consider the shift map associated to the billiard map, see Section 2, and compute the metric entropy for each ergodic measure of this shift. We must treat several cases depending on the support of the measure. If the ergodic measure has its support included in the definition set, then the method of Katok 
can be used with minor changes, see Section 3. The other case can not appear in dimension two and represent the main problem in dimension three. We treat this case by looking at the billiard orbits which pass through singularities. By a geometric argument we prove in Section 4 that the support of a such measure is the union of two sets: a countable set and a set of words whose complexity can be bounded, see Proposition \ref{entroscruc} and Lemma \ref{3entrosur}.  

If we want to generalize this result to any dimension some problems appear.
Im dimension three, we treat two cases by different methods depending on the dimension of the cells.
In dimension $d$ there would be at least $d-1$ different cases and actually we have no method for these cases.
Moreover we must generalize Lemma  \ref{eqdim3} and the followings . Unfortunately this is much harder and cannot be made with computations.

\section{Background and notations}
\subsection{Definitions}
We consider the billiard map inside a convex polyhedron $P$. This map is
defined on the set $E\subset\partial{P}\times\mathbb{PR}^3$, by
the following method:

First we define the set $E'\subset\partial{P}\times\mathbb{PR}^3$.
A point $(m,\theta)$ belongs to $E'$ if and only if one of the two following points is true: \\
$\bullet$ The line $m+\mathbb{R}^*[\theta]$ intersects an edge of
$P$, where $[\theta]$ is a vector of $\mathbb{R}^3$ which
represents $\theta$.\\

 $\bullet$ The line $m+\mathbb{R}^*[\theta]$ is included inside
 the face of $P$ which contains $m$.

 Then we define $E$ as the set
 $$E=(\partial{P}\times
 \mathbb{PR}^3)\setminus E'.$$

 Now we define the map $T$:
Consider $(m,\theta)\in E$, then we have
$T(m,\theta)=(m',\theta')$ if and only if $mm'$ is colinear to
$[\theta]$, and $[\theta']=s[\theta]$, where $s$ is the linear
reflection over the face which contains $m'$.
$$T:E\rightarrow \partial{P}\times\mathbb{PR}^3$$
$$T:(m,\theta)\mapsto (m',\theta')$$
\begin{remark}
In the following we identify $\mathbb{PR}^3$ with the unit vectors
of $\mathbb{R}^3$ ({\it i.e} we identify $\theta$ and
$[\theta]$).
\end{remark}

\begin{definition}
The set $E$ is called the phase space.
\end{definition}
\begin{figure}[h]
\includegraphics[width= 5cm]{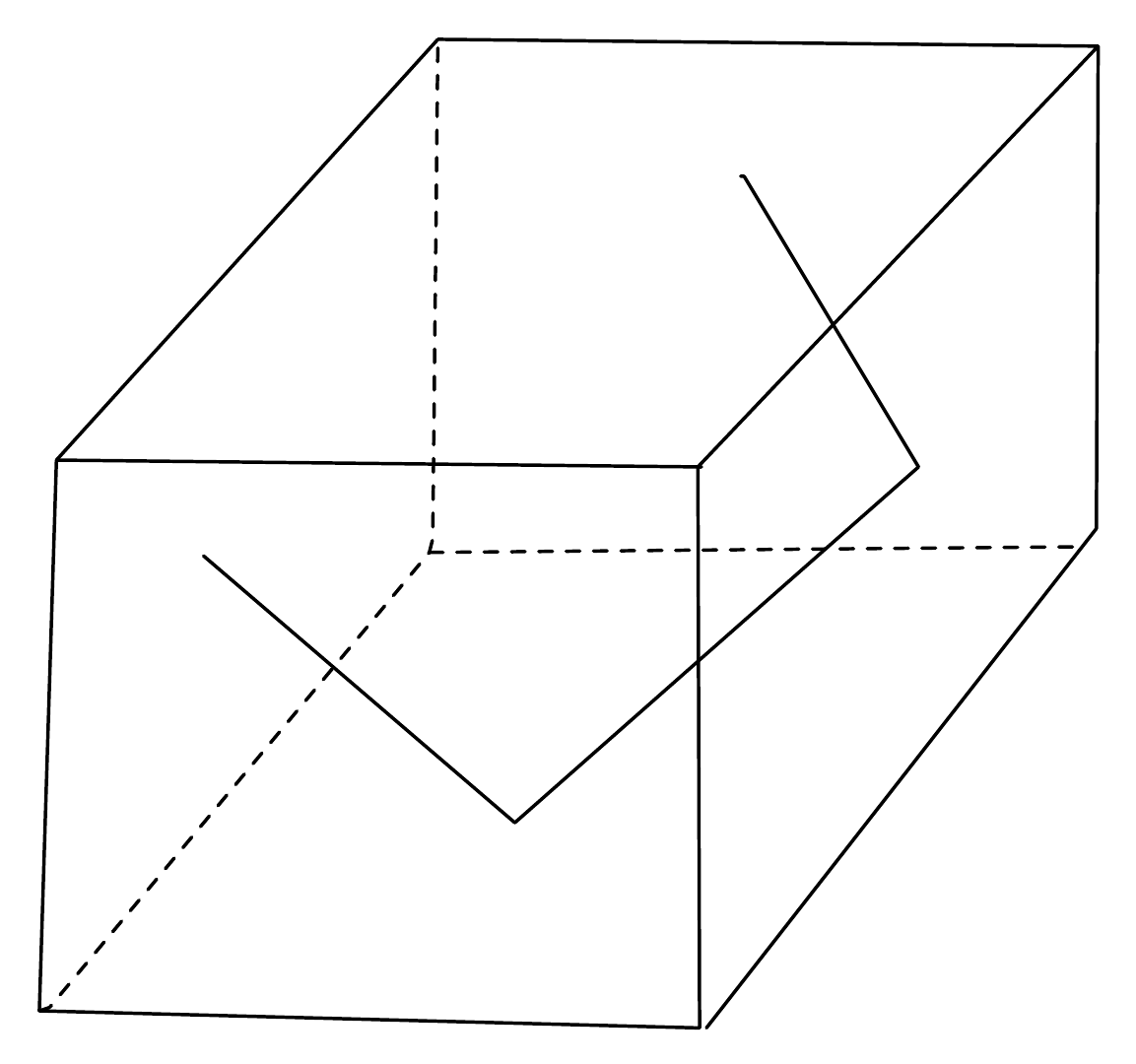}
\caption{Billiard map inside the cube}

\end{figure}

\subsection{Combinatorics}
\begin{definition}
Let $\mathcal{A}$ be a finite set called the alphabet. By a
language $L$ over $\mathcal{A}$ we mean always a factorial
extendable language: a language is a collection of sets
$(L_n)_{n\geq 0}$ where the only element of $L_0$ is the empty
word, and each $L_n$ consists of words of the form $a_1a_2\dots
a_n$ where $a_i\in\mathcal{A}$ and such that for each $v\in L_n$
there exist $a,b\in\mathcal{A}$ with $av,vb\in L_{n+1}$, and for
all $v\in L_{n+1}$ if $v=au=u'b$ with $a,b\in\mathcal{A}$ then
$u,u'\in L_n$.\\
The complexity function of the language $L$, $p:\mathbb{N}\rightarrow\mathbb{N}$ is
defined by $p(n)=card(L_n)$.
\end{definition}

\subsection{Coding}
We label each face of the polyhedron with a letter from the
alphabet $\{1\dots N\}$. Let $E$ be the phase space of the
billiard map and $d=\{d_1\dots d_N\}$ the cover of $E$ related to
the coding. The phase space is of dimension four : two coordinates
for the point on the boundary of $P$ and two coordinates for the
direction.

Let $E_0$ be the points of $E$ such that $T^n$ is defined,
continuous in a neighborhood for all $n\in\mathbb{Z}$. Denote by
$\phi$ the coding map, it means the map

$$\phi : E_0\rightarrow\{1,\dots, N\}^{\mathbb{Z}},$$
$$\phi(p)=(v_n)_\mathbb{Z},$$
where $v_n$ is  defined by $T^n(p)\in d_{v_n}$. Let $S$ denote the
shift map on $\{1\dots N\}^{\mathbb{Z}}$. We have the diagram,
\begin{equation*}
\begin{CD}
E_0     @>T>>  E_0\\
@V{\phi}VV        @VV{\phi}V\\
\phi(E_0) @>>S> \phi(E_0)
\end{CD}
\end{equation*}
with the equation $\phi\circ T=S\circ\phi.$

We want to compute the topological entropy of the billiard map.
We define the topological entropy of the billiard map as
the topological entropy of the subshift, see Definition
\ref{htop}. 

We remark that the proof of Theorem \ref{entro} given
in \cite{Ga.Kr.Tr} as a corollary of their result is not complete: They do not consider the case,
where the ergodic measure is supported on the boundary of
$\phi(E_0)$.

\subsection{Notations}
Let $\Sigma$ be the closure of $\phi(E_0)$, and consider the cover
$$d\vee T^{-1}d\vee \dots\vee T^{-n+1}d.$$
The cover $d$, when restricted to $E_0$, is a partition.
The sets of this cover are called $n$-cells.
If $v\in\Sigma$ we denote
$$\sigma_v=\bigcap_{n\in\mathbb{Z}}\overline{T^{-n}(d_{v_n}\cap E_0)}=
\displaystyle\bigcap_{n\in\mathbb{Z}}T^{-n}d_{v_n}.$$
It is the closure of the set of points of $E_0$ such that the orbit is coded by $v$. If $v\in\phi(E_0)$ then $\sigma_v$  is equal to $\phi^{-1}(v)$.
We denote $d^{-}=\displaystyle \bigvee^{\infty}_{n=0}T^{-n}d$ and
$$\sigma_v^{-}=\displaystyle\bigcap_{n\geq 0}\overline{T^{-n}(d_{v_n}\cap\phi(E_0))}
=\displaystyle\bigcap_{n\geq 0}T^{-n}d_{v_n}.$$

\begin{definition}\label{xi}
Let $\xi=\{c_1,\dots,c_k\} $ be the partition of $\Sigma$ given by
$$c_k=\overline{\phi(d_k\cap E_0)}.$$
\end{definition}
Finally we can define the topological entropy
\begin{definition}\label{htop}
Consider a polyhedron of $\mathbb{R}^3$, and $T$ the billiard map, then we define
$$h_{top}(T)=\lim_{n\rightarrow +\infty}\frac{\log p(n)}{n},$$
where $p(n)$ is the number of $n$-cells.
\end{definition}

This definition is made with the help of the following lemma which links it
to the topological entropy of the shift.
\begin{lemma}\label{entrodef}
With the same notation
$$\lim_{n\rightarrow +\infty}\frac{\log{p(n)}}{n}=h_{top}(S|\Sigma).$$
\end{lemma}
\begin{proof}
The partition $\xi$, see Definition \ref{xi}, is a topological generator of
$(S|\Sigma)$ (see \cite{Pet} for a definition), thus
$$h(S|\Sigma)=\lim_{n\rightarrow +\infty}\frac{\log{card \xi_n}}{n},$$
and we have ${\rm card}(\xi_n)=p(n).$
\end{proof}
\begin{remark}
The number of cells, $p(n)$, is equal to the complexity of the
language $\Sigma$.\\
There are several other possible definitions (Bowen definition $\dots$) but we use this one since we are interested in the complexity function of the billiard map. 
\end{remark}

\subsection{Billiard}
\subsubsection{Cell}

We denote by $\pi$ the following map:

$$\pi: \partial{P}\times\mathbb{PR}^3\mapsto \mathbb{PR}^3$$
$$\pi:(m,\theta)\rightarrow \theta.$$
Consider an infinite word $v\in\phi(E_0)$. 
\begin{definition}
We consider the elements $(m,\theta)$ of $\partial{P}\times\mathbb{PR}^3$ as 
vectors $\theta$ with base point $m$. \\
We say that $X\subset \partial{P}\times\mathbb{PR}^3$ is a strip if 
all $x\in X$ are parallel vectors whose base points form an interval.\\
We  say that $X\subset \partial{P}\times\mathbb{PR}^3$ is a tube if 
all $x\in X$ are parallel vectors whose base points form an open polygon or an open ellipse.\\
\end{definition}
Now we recall the theorem of Galperin, Kruger  and Troubetzkoy \cite{Ga.Kr.Tr},  which describe the shape of $\sigma_v^-$:
\begin{lemma}\label{dim}
Let $v\in\phi(E_0)$ be an infinite word, then there are three cases:\\
The set $\sigma_v^-$ consists of only one point.\\
The set $\sigma_v^-$ is a strip.\\
The set $\sigma_v^-$ is a tube.\\
Moreover if $\sigma_v^-$ is a tube then $v$
is a periodic word.
\end{lemma}
\begin{remark}
The preceding lemma shows that $\phi$ is not bijective on $E_0$.
\end{remark}
By the preceding lemma for each infinite word $v$ the set $\pi(\sigma_v^-)$ is unique. If the base points form an interval we say that $\sigma_v^-$ is of dimension one, and of dimension two if the base points form a polygon or an ellipse.

\begin{definition}\label{def}
As in the preceding lemma, if $v$ is an infinite word we say that
$\pi(\sigma_v^-)$ is the direction of the word.\\
Moreover if $v$ is an infinite word, we identify $\sigma_v^-$ with
the set of base points  $a$ which fulfills $\sigma_v^-=a\times \pi(\sigma_v^-)$.
\end{definition}

\subsubsection{Geometry}
First we define the rational polyhedron. Let $P$ be a
polyhedron of $\mathbb{R}^3$, consider the linear reflections
$s_i$ over the faces of $P$.
\begin{definition}
We denote by $G(P)$ the group generated by the $s_i$, and we say that $P$ is
rational if $G(P)$ is finite.
\end{definition}
In $\mathbb{R}^2$ a polygon is rational if and only if all the
angles are rational multiples of $\pi$. Thus the rational polygons
with $k$ edges are dense in the set of polygons with $k$ edges. In
higher dimension, there is no simple characterization of rational
polyhedrons, moreover their set is not dense in the set of
polyhedrons with fixed combinatorial type (number of edges,
vertices, faces).

An useful tool in the billiard study is the unfolding. When a trajectory passes
through a face, there is reflection of the line. The unfolding consists in
following the same line and in reflecting the polyhedron over the
face. For example for the billiard in the square/cube, we obtain the usual
square/cube tiling.
In the following we will use this tool, and an edge means an edge
of an unfolded polyhedron.

\subsection{Related results}
If $P$ is a rational polyhedron, then we can define the first
return map of the directional flow in a fixed direction $\omega$.
This map $T_{\omega}$ is a polygon exchange (generalization of
interval exchange). Gutkin and Haydn have shown :

\begin{theorem}\cite{Gu.Ha}
Let $P$ be a rational polyhedron and $w\in\mathbb{S}^2$ then
$$h_{top}(T_{\omega})=0.$$
Moreover if $\mu$ is any invariant measure then
$$h_{\mu}(T)=0.$$
\end{theorem}

Buzzi \cite{Bu}, has generalized this result. He proves that each
piecewise isometrie of $\mathbb{R}^n$ have zero topological
entropy. Remark that a polygonal exchange is a piecewise
isometry.
\section{Variational principle}

We use the variational principle to compute the entropy
$$h_{top}(S|\Sigma)=\displaystyle\sup_{\substack{\mu\\ ergo}}
h_{\mu}(S|\Sigma).$$ Remark that we cannot apply it to the map $T$
since it is not continuous on a compact metric space. The
knowledge of $h_\mu(T)$ does not allow to compute $h_{top}(T)$. We
are not interested in the atomic measures because the associated
system is periodic, thus their entropy is equal to zero. We split
into two cases $supp(\mu)\subset\phi(E_0)$ or not. We begin by
treating the first case which is in the same spirit as the
argument in Katok \cite{Ka}.

\begin{figure}[t]
\begin{center}
\includegraphics[width= 6cm]{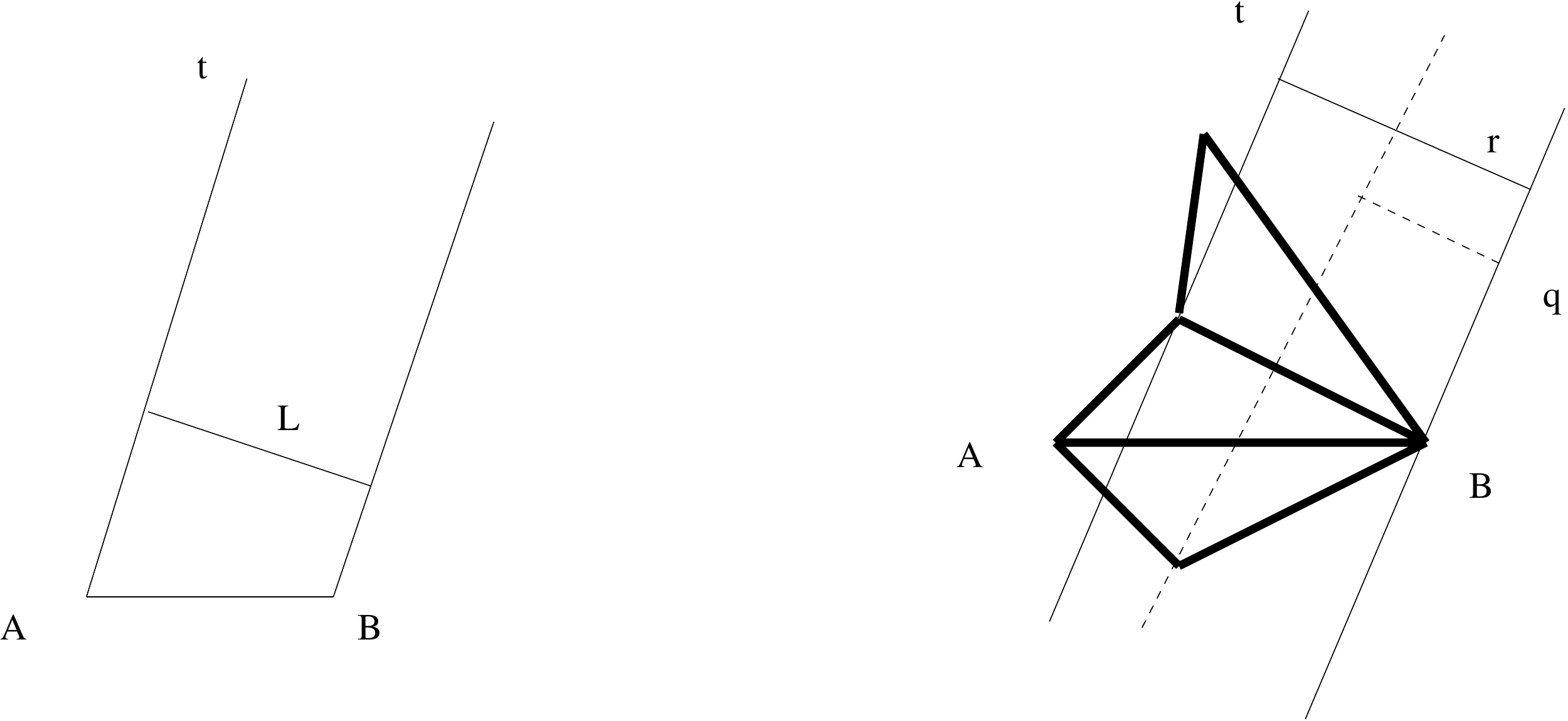}
\caption{Billiard invariant}
\label{e31}
\end{center}
\end{figure}

\begin{lemma}\label{lemka}
Let $\mu$ be an ergodic measure with support in $\phi(E_0)$. We
denote $\xi^{-}=\displaystyle\bigvee^{\infty}_{n=0}S^{-n}\xi$,
where $\xi$ is defined in Definition \ref{xi}. Up to a set of
$\mu$ measure zero we have
$$S\xi^{-}=\xi^{-}.$$
\end{lemma}

\begin{proof}
As $\mu(\phi(E_0))=1$, the cover $\xi$ can be thought as a partition of
$\phi(E_0)$. Let $v\in\phi(E_0)$, then the set
$\sigma_v^-$ can be thought as an element of $d^{-}$.
The set $\overline{\phi(\sigma_{v}^-\cap E_0)}$ coincides with the
set of $\xi^{-}$ which contains $v$.

By Lemma \ref{dim} the dimension of $\sigma_v^-$ can take three
values.

 We have $\sigma_{S^{-1}v}^-\subset T^{-1}\sigma_{v}^-$,
thus the set of $v$ such that  $\sigma_{v}^-$ is a point is
invariant by $S$. The ergodicity of $\mu$ implies that this set
either has zero measure or full measure.

Assume it is of full measure,
then $d^{-}$ is a partition of points, and same thing for
$\xi^{-}$. Then $\xi^{-}$ is a refinement of  $S\xi^{-}$ , this implies that those two sets are equal.

Assume it is of zero measure. Then by ergodicity there are two
cases : $\sigma_{v}^-$ is an interval or of dimension two for a
set of full measure.

$\bullet$ Assume $\sigma_{v}^-$ is an interval for a
full measure set of $v$.

If $\theta$ is the direction of $v$, then consider the strip
$\sigma_v^-+\mathbb{R}\theta$. Consider a line included in the plane of the strip and 
orthogonal to the axis $\mathbb{R}\theta$, and denote $L(\sigma_v^-)$ the length of
the set at the intersection of the line and the strip, see Figure
\ref{e31}.

Clearly we have $T(\sigma_{v}^-)\subset \sigma_{Sv}^-,$
thus we have $L(T\sigma_{v}^-)\leq L(\sigma_{Sv}^-).$
Since $L(T\sigma_v^-)=L(\sigma_v^-)$ we conclude that the function $L$
is a sub-invariant of $S$.

Since $\mu$ is ergodic the function $L$ is constant $\mu$ {\em a.e}.
Thus for $\mu$ {\em a.e} $v$ we obtain two intervals of same
length, one included in the other. They are equal. We deduce
$\sigma_{Sv}^-=T\sigma_{v}^-.$
This implies that $v_1,v_2,\dots$ determines $v_0$ almost surely.
It follows that
$$S\xi^-=\xi^- \mu {\it a.e}.$$

$\bullet$ If $\sigma_{v}^-$ is of dimension $2$ for a positive measure
set of $v$, by ergodicity it is of the same dimension for $\mu$ {\em a.e} $v$.
It implies that $v$ is a periodic word  $\mu$ {\em a.e}, thus $S\xi^{-}=\xi^{-} \mu$ {\em a.e}.
\end{proof}

Since $h(S,\xi)=H(S\xi^-|\xi^-)=0$ we have :

\begin{corollary}\label{supmes1}
If $supp(\mu)\subset\phi(E_0)$ then $h_{\mu}(S|\Sigma)=0$.
\end{corollary}

\section{Measures on the boundary}
We will treat the cases of ergodic measures, satisfying
$$X=supp(\mu)\subset\Sigma\setminus\phi(E_0).$$
First we generalize Lemma \ref{dim}:

\begin{lemma}\label{entro3gkt}
For a convex polyhedron, for any word
$v\in\Sigma\setminus\phi(E_0)$ the set $\sigma_v^-$ is connected
and is a strip.
\end{lemma}
We remark that Lemma \ref{entro3gkt} is the only place where we
use the convexity of $P$.

\begin{proof}
First the word $v$ is a limit of words $v^n$ in $\phi(E_0)$. Each
of these words $v^n$ have a unique direction $\theta_n$ by Lemma
\ref{dim}. The directions $\theta_n$ converge to
$\theta$, this shows that the direction of $\sigma_v^-$ is unique.
Now by convexity of $P$ the set $\sigma_v^-$ is convex as
intersection of convex sets. By definition the projection of
$\sigma_v^-$ on $\partial{P}$ is included inside an edge, thus it
is of dimension less than or equal to one. This implies that the set is an
interval or a point.
\end{proof}

{\em A priori} there are several cases as $dim\sigma_v^-$ can be
equal to $0$ or $1$. We see here a difference with the polygonal
case. In this case the dimension was always equal to zero.

\subsection{Orbits passing through several edges}
In this paragraph an edge means the edge which appears in the
unfolding of $P$ corresponding to $v$. We represent an edge by a
point and a vector. The point is a vertex of a copy of $P$ in the
unfolding and the vector is the direction of the edge. We consider
two edges $A,B$ in the unfolding. Consider $m\in A$  and a
direction $\theta$ such that the orbit of $(m,\theta)$ passes
through an edge. We identify the point $m$ with the distance
$d(m,a)$ if $a$ is one endpoint of the edge $A$. Moreover we denote by $u$ an unit 
vector colinear to the edge $A$.

\begin{lemma}\label{entro3eq}
The set of $(m,\theta), m\in A_0$  such that the orbit of $(m,\theta)$ passes through an edge $A_1$ satisfies either

(i) $(m,\theta)$ is in the line or plane which contains $A_0,A_1$.
Then there exists an affine map $f$ such that $f(\theta) = 0$.

or

(ii) there exists a map $F :\mathbb{R}^3 \to \mathbb{R}$  such
that $m=F(\theta)$ (it is the quotient of two linear polynomials).
Moreover the map $(A_0,A_1)\mapsto F$ is injective.
\end{lemma}
\begin{remark}
The case where $A_0,A_1$ are colinear is included in the first case. In this case there are 
two equations of the form $f(\theta)=0$ but we only use one of them.
\end{remark}
\begin{proof}
Consider the affine subspace generated by the edge $A_0$ and the
line $m+\mathbb{R}\theta$. There are two cases :

$\bullet\quad A_1\in Aff(A_0,m+\mathbb{R}\theta)$. Assume
$A_0,A_1$ are not colinear, then the affine space generated by
$A_0,A_1$  is of dimension two (or one), and several points $m$
can be associated to the same direction $\theta$. In the case it
is of dimension 2, $\theta$ is in the plane which contains
$A_0,A_1$. Then there exists an affine map $f$ which gives the
equation of the plane and we obtain $f(\theta)=0$.

$\bullet\quad A_1\notin Aff(A_0,m+\mathbb{R}\theta)$, then the
space $Aff(A_0,A_1)$ is of dimension three. If the direction is
not associated to a single point then the edges $A_0,A_1$ are
coplanar. Thus in our case the direction is associated to a single
point $m$. There exists a real number $\lambda$ such that
$m+\lambda\theta\in A_1$. Since $A_1$ is an edge, it is the
intersection of two planes (we take the planes of the two faces of
the polyhedron). We denote the two planes by the equations
$h=0;g=0$ where $h,g\quad \mathbb{R}^3\rightarrow \mathbb{R}$. We
obtain the system
$$h(m+\lambda\theta)=0,$$
$$g(m+\lambda\theta)=0.$$
Here  $h(x)=<v_h,x>+b_h$ where $v_h$ is a vector and
$<\cdot,\cdot>$ is the scalar product and similarly for $g$. Then
we write $h(m)=<v_h,mu>+b_h=m<v_h,u>+b_h$, we do the same thing
for $g$. Since $A_0,A_1$ are not coplanar the terms
$<v_g,\theta>,<v_h,\theta>$ are non null, thus we obtain the
expression for $\lambda$ :
$$\lambda=\frac{-b_h-m<v_h,u>}{<v_h,\theta>}=
\frac{-b_g-m<v_g,u>}{<v_g,\theta>}.$$ For a fixed $\theta$, there
can be only one point $m\in A_0$ which solves this equation,
otherwise we would be in case $(i)$. Thus we find $m=F(\theta)$
where $F$ is the quotient of two linear polynomials :
$$m=\frac{b_g<v_h,\theta>-b_h<v_g,\theta>}{<v_h,u><v_g,\theta>-<v_g,u><v_h,\theta>}\quad (*).$$
Note that $F$ does not depend on
the concrete choices of the planes $h,g$, but only on the edges $A_0,A_1$.

We prove the last point by contradiction. If we have the same
equation for two edges, it means that all the lines which pass
through two edges pass through the third. We claim it implies
that the three edges $A_0,A_1,A_2$ are coplanar : the first case
is when $A_1,A_2$ are coplanar. Then the assumption implies that
the third is coplanar, contradiction. Now assume that the three
edges are pairwise not coplanar. Indeed consider a first line
which passes through the three edges. Call $m$ the point on $A_0$,
and $u$ the direction. Now consider a line which contains $m$ and passes through
$A_1$ with a different direction. Those two lines intersect $A_1$,
thus $m$ and the two lines are coplanar. Since $A_2$ is not
coplanar with $A_0$, both lines can not intersect $A_2$,
contradiction. To finish consider the case when two edges are colinear but the third one is not colinear with either of the other two. This case can be reduced to the first case by looking at the first and third edges.   
\end{proof}

\begin{lemma}\label{entrogeom}
Consider two edges $A_0,A_i$ which give the equation $m=F_i(\theta)$. Denote by $p_i$ a
point on $A_i$ and $x_i$ the direction of the line $A_i$. Then we have
$$F_i(\theta)=\frac{<p_i\wedge x_i,\theta>}{<u\wedge x_i,\theta>},$$
where $u$ is an unit vector colinear to the edge $A_0$.
\end{lemma}
\begin{proof}
By Lemma \ref{entro3eq} each $F_i$ is the quotient of two
polynomials. Consider the denominator of $F_i$ as function of
$\theta$ ( we use the notations of the preceding proof). By  equation $(*)$ we obtain: 
$$F_i(\theta)=\frac{N(\theta)}{D(\theta)},$$
$$D(\theta)= -<v_{h_i},u><v_{g_i},\theta>+<v_{g_i},u><v_{h_i},\theta>.$$ We
remark for the map $F_i$ that
$$-<v_{h_i},u>v_{g_i}+<v_{g_i},u>v_{h_i},$$ is orthogonal to $u$ and to $x_i$. Thus this vector is colinear to
$u\wedge x_i$ :
$$-<v_{h_i},u>v_{g_i}+<v_{g_i},u>v_{h_i}=C_iu\wedge x_i.$$

Consider the numerator $(b_{h_i}v_{g_i}-b_{g_i}v_{h_i},\theta)$ of
$F_i$. The scalar product of $b_{h_i}v_{g_i}-b_{g_i}v_{h_i}$ with
$x_i$ is null, moreover the scalar product with $p_i$ equals again
zero by definition of $v_{g_i},b_{g_i},v_{h_i},b_{h_i}$. Thus we
obtain :
\begin{equation}
b_{h_i}v_{g_i}-b_{g_i}v_{h_i}=C'_ip_i\wedge x_i,
\end{equation}
and :
$$F(\theta)=\frac{C'_i}{C_i}\frac{<p_i\wedge x_i,\theta>}{<u\wedge x_i,\theta>}.$$
We claim that $C_i=C'_i=1$. We can choose the vectors $v_{g_i},v_{h_i}$ such that they are orthogonal
and of norm 1. Then $x_i$ is colinear to
$v_{g_i}\wedge v_{h_i}$ and is of norm one, thus if we choose the proper orientation of $x_i$ they are equal.
Then we can have
$$-<v_{h_i},u>v_{g_i}+<v_{g_i},u>v_{h_i}=u\wedge(v_{g_i}\wedge v_{h_i})=u\wedge x_i.$$
Thus we deduce $K_i'=1$.

Now we compute the norm of the vector of the numerator
$|b_{h_i}v_{g_i}-b_{g_i}v_{h_i}|^2=b_{h_i}^2+b_{g_i}^2$. By
definition of $b_{g_i}, b_{h_i},p_i$ we obtain
$$b_{g_i}=-<v_{g_i},p_i>; b_{h_i}=-<v_{h_i},p_i>.$$
Thus we have $|b_{h_i}v_{g_i}-b_{g_i}v_{h_i}|^2=<v_{g_i}|p_i>^2+<v_{h_i}|p_i>^2$. Moreover by definition we have
that $x_i=v_{g_i}\wedge v_{h_i}$ this implies
that $|p_i\wedge x_i|^2=<v_{g_i},p_i>^2+<v_{h_i},p_i>^2$. Finally we deduce $$|p_i\wedge x_i|^2(C'_{i})^2=
|p_i\wedge x_i|^2.$$
\end{proof}

\begin{lemma}\label{eqdim3}
Consider three edges $A_0,A_1,A_2$ such that $dim Aff(A_i,A_j)=3$ for all $i,j$. Then the sets of lines $d$ which
pass through
$A_0,A_1,A_2$ is contained in a surface which we call $S(A_0,A_1,A_2)$. Consider an orthonormal basis such that the
direction $u$ of $A_0$ satisfies
$u=\begin{pmatrix}1\\0 \\ 0\end{pmatrix}$. If we call $(P_1,P_2,P_3)$ the coordinates of a point on this surface, then

$(i)$ the equation of the surface can be written as $P_1=f(P_2,P_3)$, where $f$ is a polynomial.

$(ii)$ there exists $N\leq 4$ such that any line which is not contained in $S$ intersects $S$ at most $N$ times.
\end{lemma}
\begin{proof}
Consider a line $d=m+\mathbb{R}\theta, m\in A_0$ which passes through $A_1,A_2$. By Lemma \ref{entro3eq} we obtain
two equations $m=F_i(\theta)$. Then Lemma
\ref{entrogeom} implies that $F_i(\theta)=\frac{\sum a_{j,i}\theta_j}{\sum_{j=2}^{3}b_{i,j}\theta_j}$.
Now call $P_i$ the coordinates of a point $P$ on $d$. We have $P=m+\lambda\theta$, thus we obtain

$$\begin{cases}P_1=\frac{a_{1}\theta_1+a_2\theta_2+a_3\theta_3}{b_2\theta_2+b_3\theta_3}+\lambda\theta_1\\ P_2=\lambda
\theta_2\\ P_3=
\lambda\theta_3\\ (F_1-F_2)(\theta)=0
\end{cases}$$
where $a_j=a_{j,1}$ and $b_j=b_{1,j}$.\\
$\bullet$ First case $P_2\neq 0$. This is equivalent to $\theta_2\neq 0$.

$$\begin{cases}P_1=\frac{a_{1}\theta_1+a_2\theta_2+a_3\theta_3}{b_2\theta_2+b_3\theta_3}+\lambda\theta_1\\ P_2=\lambda
\theta_2\\ \theta_3=\frac{P_3}
{P_2}\theta_2\\ (F_1-F_2)(\theta)=0
\end{cases}$$

$$\begin{cases}P_1=\frac{a_{1}\theta_1+\theta_2(a_2+a_3\frac{P_3}{P_2})}{\theta_2(b_2+\frac{P_3}{P_2})}+
P_2\frac{\theta_1}{\theta_2}\\ P_2=\lambda \theta_2\\
\theta_3=\frac{P_3}{P_2}\theta_2 \\ (F_1-F_2)(\theta)=0
\end{cases}$$
  $$\begin{cases}P_1=\frac{a_{1}}{(b_2+\frac{P_3}{P_2})}\frac{\theta_1}{\theta_2}+\frac{a_2+
  a_3\frac{P_3}{P_2}}{b_2+\frac{P_3}{P_2}}+
P_2\frac{\theta_1}{\theta_2}\\ P_2=\lambda
\theta_2\\
\theta_3=\frac{P_3}{P_2}\theta_2 \\ (F_1-F_2)(\theta)=0
\end{cases}$$
 $$\begin{cases}P_1=(\frac{a_{1}}{b_2P_2+P_3}+1)P_2\frac{\theta_1}{\theta_2}+\frac{a_2P_2+a_3P_3}{b_2P_2+P_3}\\
 P_2=\lambda
\theta_2\\
\theta_3=\frac{P_3}{P_2}\theta_2 \\ (F_1-F_2)(\theta)=0
\end{cases}$$
Now the equation $(F_1-F_2)(\theta)=0$ can be written as
$$(\sum_{j=1}^3 a_{j}\theta_j)(\sum_{j=2}^{3}b'_{j}\theta_j)=(\sum_{j=1}^3 a'_{j}\theta_j)(\sum_{j=2}^{3}b_{j}\theta_j),$$
where $a'_j=a_{j,2}$ and $b'_j=b_{2,j}$.\\
$$(a_{1}\theta_1+a_{2}\theta_2+a_{3}\theta_3)(b'_{2}\theta_2+b'_{3}\theta_3)=
(a'_{1}\theta_1+a'_{2}\theta_2+a'_{3}\theta_3)(b_{2}\theta_2+b_{2}\theta_3).$$
With the equation $\theta_3=\frac{P_3}{P_2}\theta_2$ we obtain an equation of the following form.

\begin{gather*}
(a_1\theta_1P_2+(a_2P_2+a_3P_3)\theta_2)(b'_2P_2+b'_3P_3)=\\
(a'_1\theta_1P_2+(a'_2P_2+a'_3P_3)\theta_2)(b_2P_2+b_3P_3).\\
(a_1\theta_1/\theta_2P_2+(a_2P_2+a_3P_3))(b'_2P_2+b'_3P_3)=\\
(a'_1\theta_1/\theta_2P_2+(a'_2P_2+a'_3P_3))(b_2P_2+b_3P_3).
\end{gather*}
Thus we obtain the value of $\frac{\theta_1}{\theta_2}$.
\begin{gather*}
\theta_1/\theta_2[a_1(b'_2P_2+b'_3P_3)-a'_1(b_2P_2+b_3P_3)]P_2=\\
(a'_2P_2+a'_3P_3)(b_2P_2+b_3P_3)-(a_2P_2+a_3P_3)(b'_2P_2+b'_3P_3).
\end{gather*}

If the coefficient of $\frac{\theta_1}{\theta_2}$ is null we obtain an equation of the form $P_2=KP_3$. This
implies that $P$ is on a plane. It is impossible
since the lines $A_i$ are non coplanar. Thus we can obtain the value of $\frac{\theta_1}{\theta_2}$.
Then the first line of the system gives an equation of the form
$$f(P_2, P_3)=P_1,$$
where $f$ is a homogeneous rational map of twp variables.

$\bullet$ Second case $P_2=0$. We obtain

$$\begin{cases}P_1=\frac{a_{1}\theta_1+a_3\theta_3}{b_3\theta_3}+\lambda\theta_1\\ P_3=\lambda\theta_3\\
(F_1-F_2)(\theta)=0
\end{cases}$$
Remark that $P_3\neq0$. Indeed if not the direction is included in $A_0$. Thus the system becomes
$$\begin{cases}P_1=\frac{a_{1}\theta_1+a_3\theta_3}{b_3\theta_3}+\lambda\theta_1\\ \lambda=P_3/\theta_3\\
(F_1-F_2)(\theta)=0
\end{cases}$$
$$\begin{cases}P_1=\frac{a_{1}\theta_1+a_3\theta_3}{b_3\theta_3}+P_3/\theta_3\theta_1\\ P_3/\theta_3=\lambda\\
(F_1-F_2)(\theta)=0
\end{cases}$$
And the equation $(F_1-F_2)(\theta)=0$ gives as in the first case the values of $\frac{\theta_1}{\theta_3}$.

$\bullet$ Now consider a transversal line $d'$. A point on this line depends on one parameter. If the point is
on the surface, the parameter verifies a polynomial
equation of degree four, thus there are a bounded number of solutions.
\end{proof}

\begin{corollary}\label{entro3inde}
Consider four edges $A_0,A_1,A_2,A_3$ two by two non coplanar such that $A_3\notin S(A_0,A_1,A_2)$. Then the maps
$F_1-F_2, F_1-F_3$ are linearly
independent.
\end{corollary}
\begin{proof}
We make the proof by contradiction. If the maps $F_1-F_2,F_1-F_3$
are linearly dependent, it means that $F_3$ is a linear
combination of $F_1,F_2$. It implies that the system
$\begin{cases}m=F_1(\theta)\\m=F_2(\theta)\\m=F_3(\theta)\end{cases}$
is equivalent to
$\begin{cases}m=F_1(\theta)\\m=F_2(\theta)\end{cases}$. Thus each
line which passes through $A_0,A_1,A_2$ must passes through $A_3$.
By preceding Lemma it implies that $A_3$ is in $S(A_0,A_1,A_2)$,
contradiction.
\end{proof}

\subsection{Key point}

\begin{lemma}\label{3Ka}
Consider a point $(m,\theta)\in \overline{E_0}$; then the set of words $v$ such that $(m,\theta)\in
\sigma_v^-$ is at most countable.
\end{lemma}
For the proof we refer to \cite{Ka}. This proof does not depend on the dimension.

\subsubsection{Definitions}
For a fixed word $v\in\Sigma\setminus\phi(E_0)$, the set
$\sigma_v^-$ is of dimension 0 or 1 and the direction $\theta$ is
unique, see Lemma \ref{entro3gkt}. Fix a word
$v\in\Sigma\setminus\phi(E_0)$, we will consider several cases:

$\bullet$ First $\sigma_v^-$ is an interval with endpoints $a,b$. For any
$m\in]a,b[$ we consider the set of discontinuities met in
the unfolding of $(m,\theta)$. 
This set is independent of
$m\in ]a,b[$ since $\sigma_v^-$ is an interval. We denote it $Disc(v,int)$. 
 If the endpoint $a$ (resp. $b$) is included in the interval then the orbit of
$(m,\theta)$ can meet other discontinuities. 
We call $Disc(v,a)$ (resp. $Disc(v,b)$) the set of those discontinuities.

$\bullet$ If $\sigma_v^-$ is a point it is the same method  as $Disc(v,int)$, we denote the set of
discontinuities by $Disc(v,int)$.

Here there are two sorts of
discontinuities.  First the singularity is a point of the boundary
of a face whose code contributes to $v$. 
Then the orbit is not transverse to the edge. 
Secondly they meet in the transversal
sense. If the orbit is included in an edge, then the
discontinuities met are the boundary points of that edge (and
similarly if the orbit is in a face).

\begin{definition}
Let $V=\Sigma\setminus\phi(E_0)$ and $X\subset V$ be the set of
$v\in V$ such that the union of the elements $A_i$ of
$Disc(v,int),Disc(v,a),Disc(v,b)$ are contained in a finite union
of hyperplanes and of surfaces $S(A_0,A_1,A_2)$.
\end{definition}
Suppose $v\in X$. Let $N(\sigma_v^-)$ be the number of planes
containing $Disc(v)$ if $\sigma_v^-$ is a point or $Disc(v,a)$ or
$Disc(v,b)$ if $\sigma_v^-$ is an interval.

In the following Lemma the function $L$ refers to the width of the strip of singular orbits as it does
in the proof of Lemma \ref{lemka}.
\begin{lemma}\label{entroprobpaq}
Suppose $\mu$ is  an ergodic measure with support in $\Sigma \backslash\phi(E_0)$. Then

$(i)$ there exists a constant $L$ such that $L(\sigma_v^-)=L$ for
$\mu$-a.e. $v\in \Sigma$ and thus for $\mu$-a.e $v,w \in \Sigma$
if $w_i = v_i$ for $i \ge 0$ then $\sigma_w = \sigma_v$.

$(ii)$ there exists a constant $N$ such that $N(\sigma_v^-)=N$ for $\mu$-a.e $v\in\Sigma$.
\end{lemma}
\begin{proof}
$(i)$ If $\sigma_v^-$ is a point then there is nothing to show.
Let $L(\sigma_v^-)$ be as before.  We have $L(\sigma_v^-)\leq
L(\sigma_{S(v)}^-)$.  Since $S$ is ergodic, $L$ is constant almost
everywhere. Thus $L(\sigma_v) = L(\sigma_v^-)$ thus
$\sigma_v=\sigma_v^-$. The same holds for $w$, thus since
$\sigma_w^-=\sigma_v^-$ we have $\sigma_v= \sigma_w$.

$(ii)$ We have $N(\sigma_v^-)\leq N(\sigma_{Sv}^-)$, thus the lemma follows since $S$ is ergodic.
\end{proof}
Let $D$ stand for $Disc(v,int),Disc(v,a),\text{or}\quad Disc(v,b)$.
\begin{remark}\label{entrorem}
For two sets $A_i,A_j\in D$ the relation $dim Aff(A_i,A_j)=2$ is a transitive relation. Indeed consider three sets $A_i,A_j,A_k$ such that
$A_i\sim A_j$, and $A_j\sim A_k$. Since the line $m+\mathbb{R}\theta$ passes through $A_i,A_j,A_k$, we deduce $A_i\sim A_k$.
\end{remark}
Then we can show

\begin{proposition}\label{entroscruc}
 The set $V\setminus X$ is at most countable.
\end{proposition}
 \begin{proof}
 Let $v\in V$. Lemma \ref{entro3eq} implies that we have for each pair of discontinuities an equation $m=F(\theta)$ or $f(\theta)=0$.
 Denote the set $D$ by $A_0,\dots,A_n,\dots$.
 Either there exist discontinuities $A_{i_0},A_{i_1},A_{i_2},A_{i_3}$, such that the equations related to
 $(A_{i_0},A_{i_j})$, for all $j\leq 3$, are of the form
 $m=F(\theta)$ or not. In the following we will assume, for simplicity, that these
 three discontinuities (if they exist) are denoted by $A_0,A_1,A_2,A_3$.

$\bullet$ First assume it is not the case. Then for any subset of
$D\setminus\{A_0,A_1,A_2,A_3\}$ two elements give equations of the
form $f(\theta)=0$. By Remark \ref{entrorem} all the
discontinuities in the set $D\setminus\{A_0,A_1,A_2,A_3\}$ are in
a single hyperplane. Thus all the discontinuities of $D$ are in a
finite union of hyperplanes. We do the same thing for $Disc(v,a)$
and $Disc(v,b)$. We conclude $v\in X$.

$\bullet$ Now we treat the case where we obtain at least three equations of the form $m=F(\theta)$ for some choice of $(m,\theta)$.

Corollary \ref{entro3inde} shows that two such equations are
different since the discontinuities are not in the union of
surfaces. Thus consider the  three first equations
$m=F(\theta)=G(\theta)=H(\theta)$. It gives two equations
$(F-G)(\theta)=(F-H)(\theta)=0$.  Those two equations are
different by Corollary \ref{entro3inde}, since $F,G,H$ are
different. We deduce that the direction $\theta$ is solution of a
system of two independant equations, thus it is unique. We remark
that the vertices which appear in unfolding have their coordinates
in a countable set $\mathcal{C}$. Indeed we start from a finite
number of points corresponding to the vertices and at each step of
the unfolding we reflect them over some faces of $P$. Thus at each
step there are a finite set of vertices. Moreover the coefficients
of the edges are obtained by difference of coordinates of
vertices. By the same argument the coefficients of cartesian
equations of the hyperplanes which contains faces live in a
countable set $\mathcal{C}$. There are only a countable collection
of functions $m=F(\theta)$ which arise. Thus the solution $\theta$
corresponding to the equations $m=F(\theta)=G(\theta)=H(\theta)$
lives in a countable set. It determines $(m,\theta)$. The number
of words associated to the orbit of $(m,\theta)$ is countable by
Lemma \ref{3Ka}. Thus the set of such words is countable.
\end{proof}

\begin{figure}[t]
\begin{center}
\includegraphics[width= 4cm]{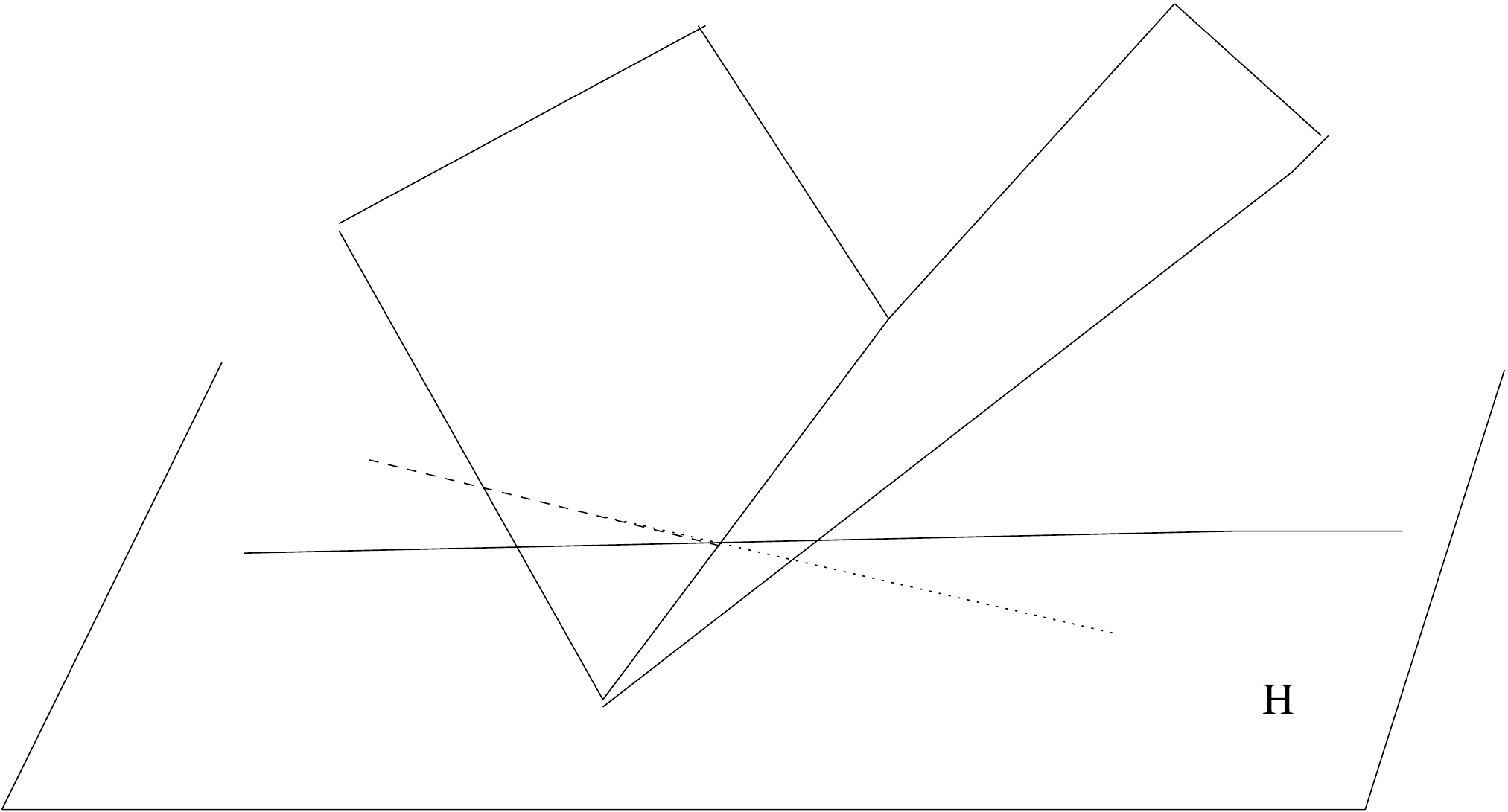}
\caption{Coding of a word}
\label{entrofighyper}
\end{center}
\end{figure}
\section{Proof of Theorem \ref{entro}}

\begin{lemma}\label{3entrosur}
Suppose that $\mu$ is an ergodic measure supported in $\Sigma \backslash \phi(E_0)$ such that $\mu(X) = 1$. Then
$h_{\mu}(S) = 0$.
\end{lemma}
\begin{proof}
By Lemma \ref{entroprobpaq} we can assume there is a constant $L \ge 0$ such that
$L(\sigma_v^-) = L$. Suppose first that $L > 0$.
Suppose $v \in \hbox{support}(\mu)$.
This implies that $Disc(v,int)$ is contained in a single plane.
If $w \in \hbox{support}(\mu)$ satisfies $w_i = v_i$ for $i \ge 0$
then $Disc(w,int)$ is contained in the same plane.  Each trajectory
in $\phi(E_0)$ which approximates the future of $v$ cuts this plane
in a single point.  Consider these sequence of approximating trajectories
which converges to $(m,\theta)$. The limit of these trajectories cuts
the surface at one (or zero) points.
The point where it cuts the surface determines
the backwards unfolding, and thus the backwards code.  Thus if we ignore
for the moment the boundary discontinuities the
knowing the future $v_0,v_1,v_2,\dots$ determines $O(n)$ choices of
the past $v_{-n},\dots,v_{-1}$.

The boundary discontinuities and the case $L(\sigma_v^-) = 0$ are
treated analogously.  Let $(m,\theta) = \sigma_v^-$ (or one of the
boundary points of $\sigma_v^-$ in the case above). By Lemma
\ref{entroprobpaq} we can assume that $Disc(v,m)$ is contained in
$N$ planes, and that if $w \in \hbox{support}(\mu)$ satisfies $w_i
= v_i$ for $i \ge 0$ then $Disc(w,int)$ is contained in the same
planes. Arguing as above, the point where an approximating orbit
cuts these planes determines the past.  Thus the future
$v_0,v_1,v_2,\dots$ determines $O(n^N)$ choices of the past
$v_{-n},\dots,v_{-1}$. Since
$\displaystyle\lim_{n\rightarrow+\infty}\frac{\log{n^N}}{n}=0$ we
deduce the result.
\end{proof}

The preceding lemma and proposition allow to conclude
\begin{corollary}\label{supmes2}
Let $\mu$ an ergodic measure with support in $\Sigma\setminus\phi(E_0)$, then
$$h_{\mu}(S)=0.$$
\end{corollary}
\begin{proof}
This follows immediately from Lemma \ref{3entrosur} and Proposition \ref{entroscruc}.
\end{proof}
Lemma \ref{entrodef} reduces the problem to the computation of $h_{top}(S|\Sigma)$.
Moreover we have
$$h_{top}(S|\Sigma)=\sup_{\substack{\mu \\ ergo, \\
supp(\mu)\subset\phi(E_0)}}h_{\mu}(S|\Sigma)+\sup_{\substack{\mu \\ ergo, \\
supp(\mu)\subset\Sigma\setminus\phi(E_0)}}h_{\mu}(S|\Sigma),$$
then Corollaries \ref{supmes1} and \ref{supmes2} imply:
$$h_{top}(S|\Sigma)=0.$$

\bibliographystyle{alpha}
\bibliography{bibio}
\end{document}